\documentclass[12pt]{article}

\usepackage{amsfonts}
\usepackage{amsthm}
\usepackage{amsmath}
\usepackage[utf8]{inputenc}
\usepackage{graphicx}
\usepackage{fancyhdr}
\usepackage{hyperref}
\usepackage{geometry}
\usepackage{enumerate}
\usepackage{amssymb}
\usepackage{xcolor}
\usepackage{authblk}
\usepackage{cite}

\newcommand{\R}{{\mathbb R}}
\newcommand{\D}{{\mathbb D}}
\newcommand{\N}{{\mathbb N}}
\renewcommand{\k}{{\bf k}}
\newcommand{\e}{{\bf e}_1}
\newcommand{\ed}{{\bf e}_2}

\newcommand{\Log}{\operatorname{Log}}

\newcommand{\calF}{\mathcal{F}}

\newcommand{\calH}{\mathcal{H}}

\newcommand{\frakG}{\mathfrak{G}}
\newcommand{\frakK}{\mathfrak{K}}

\newcommand{\frakP}{\mathfrak{P}}
\newcommand{\frakQ}{\mathfrak{Q}}

\theoremstyle{plain}
\newtheorem{theorem}{Theorem}[section]

\newtheorem{lemma}[theorem]{Lemma}

\theoremstyle{definition}
\newtheorem{definition}[theorem]{Definition}
\newtheorem{example}[theorem]{Example}

\theoremstyle{remark}
\newtheorem{remark}[theorem]{Remark}

\newcommand{\wt}{\widetilde}

\definecolor{org1}{RGB}{67, 132, 252}
\hypersetup{
  colorlinks=true,
  allcolors=blue
}

\begin{document}

\title{Extensions of the Shannon Entropy and the Chaos Game Algorithm to Hyperbolic Numbers Plane}
\author[$\dagger$]{G. Y. Téllez-Sánchez}
\author[$\star$]{J. Bory-Reyes}
\date{}
  
\affil[$\dagger$]{
\footnotesize
Escuela Superior de Fisica y Matem\'aticas

Instituto Polit\'ecnico Nacional

Edif. 9, 1er piso, U.P. Adolfo L\'opez Mateos

07338, Mexico City,  MEXICO

gtellez.wolf@gmail.com
\vspace{3mm}
}
 
\affil[$\star$]{
\footnotesize
Escuela Superior de Ingenieria Mec\'anica y El\'ectrica

Instituto Polit\'ecnico Nacional

Edif. 5, 3er piso, U.P. Adolfo L\'opez Mateos

07338, Mexico City, MEXICO

juanboryreyes@yahoo.com}

\maketitle

\begin{abstract}
In this paper we provide extensions to hyperbolic numbers plane of the classical Chaos game algorithm and the Shannon entropy. Both notions connected with that of probability with values in hyperbolic number, introduced by D. Alpay et al \cite{eluna}. Within this context, particular attention has been paid to the interpretation of the hyperbolic valued probabilities and the hyperbolic extension of entropy as well.
\vspace{0.3cm}

\small{
} 
\end{abstract}

\section{INTRODUCTION}
In 1848 James Cockle \cite{cockle} introduced the Tessarines, which form an algebra isomorphic to the bicomplex numbers \cite{segre}. Hyperbolic numbers, also known as the Lorentz numbers or double numbers are a particular type of the Tessarines. The complex numbers and the hyperbolic numbers are two-dimensional vector spaces over the real numbers, so each can identified with points in the plane $\R^2$.

The hyperbolic numbers can be considered as an hybrid between the real and complex numbers, in part because they behave in more similar nature to the real numbers, meanwhile having properties related with the complex numbers. Due to the significance of this fact for our purposes, we have compiled some basic notation and terminology in Section \ref{hypnum}. For a deeper discussion of this phenomenon we refer the reader to \cite{lusha, vignaux, sob1, sob2}.

Given a finite discrete probability distribution, the amount of uncertainty of the distribution, that is, the amount of uncertainty concerning the outcome of an experiment, the possible results of which have the given probabilities is a very important quantity called the entropy of the distribution and plays a key role in many aspects of statistical mechanics and information theory.

The idea of measure the average missing information for a given physical system associated with each possible data value, as the negative logarithm of the probability mass function for the value, was presented by Claude Elwood Shannon in his seminal paper \cite{shannon}. He proved (original proof goes back to the work of Erd\"os in \cite{erdos}) that logarithm is the only function fulfills three basic requisites for defining an entropy-like measure (see below). Section \ref{hypprobent} discusses these assumptions in the hyperbolic numbers case. 

During the past few years, several extensions of Shannon's original work have resulted in many alternative measures of entropy. For instance, by relaxing the third of Shannon's requirements, that of additive property, Alfréd Renyi \cite{renyi} was able to extend Shannon entropy to a continuous parametric family of entropy measures. Another way of stating a refinement of the requirements for entropy measures and so, the systems of postulate characterizing axiomatically these entropies, were given by different authors, see for instance \cite{fadeev, navasa, vest, acda, ebsasa, a1, a2}. 

Since the general acceptance that a mathematical expression of entropy measures are interesting enough, several theoretical attempts have been made to propose different examples of such expression (see for instance \cite{ubriaco, karci, T, E, EM, fu}). In the vast majority of cases, the key to extending entropies expressions has been the introduction of a revised logarithm.

Proceeding further in this direction, in Section \ref{hypprobent}, associated with sets of hyperbolic valued probabilities two extensions of Shannon entropy (called weak and strong hyperbolic entropy) are introduced by exploiting the hyperbolic logarithm. A generalization of probability space and basic properties in hyperbolic numbers plane is due to D. Alpay et al \cite{eluna}. 

The construction of Cantor like sets was generalized into the context of hyperbolic numbers in \cite{cantls, tebo}, where the partial order structure of such numbers was utilized. With the notion of iterated function systems (IFS for short) on hyperbolic numbers plane (see \cite{tebo2}) associating probabilities to every function of the IFS (as Barnsley made in \cite{barnsley}) a chaos game algorithm for hyperbolic numbers is proposed in Section \ref{hypchaos}, where the interpretation of the hyperbolic-valued probabilities meaning play a crucial role.

\section{HYPERBOLIC NUMBERS}
\label{hypnum}
The set of hyperbolic numbers or hyperbolic numbers plane is defined as the ring
\begin{displaymath}
\D := \R[\k] = \{ a + kb \ | \ a, b \in \R\},
\end{displaymath}
where $\k \not\in \R$ and $\k^2 = 1$. It is a commutative ring with zero divisor. Two important elements are 
\begin{displaymath}
\e = \frac{1}{2}(1 + \k), \quad \ed = \frac{1}{2}(1 - \k). 
\end{displaymath}
These elements are zero divisors, idempotent,  mutually annullable ($\e \ed = 0$) and can generate the whole hyperbolic plane as
\begin{displaymath}
\D = \R\e + \R\ed.
\end{displaymath}
The set of zero divisors is well located as $\frakG = \R\e \cup \R\ed$ and $\frakG_0 = \frakG \cup \{0\}$. The correspondence 
\begin{equation}\label{correspondence}
x \mapsto x\e + x\ed =: \wt{x}
\end{equation}
embeds $\R$ in $\D$.

The hyperbolic numbers are equipped with a partial order $\preceq$ defined by 
\begin{displaymath}
\alpha \preceq \beta\    \mbox{if and only if}\    a_1 \leq b_1\ \mbox{and}\   a_2 \leq b_2,
\end{displaymath}
where $\alpha = a_1\e + a_2\ed$, $\beta = b_1\e + b_2\ed$. Consequently, an hyperbolic interval may be defined as the set
\begin{displaymath}
[\alpha, \beta]_{\k} := \{ \xi \in \D \  |\  \alpha \preceq \xi \preceq \beta\}.
\end{displaymath}

Additionally the set of positive hyperbolic numbers is defined as
\begin{displaymath}
\D^+ := \{ \xi \in \D \ | \ \xi \succeq 0 \}.
\end{displaymath}
This set is closed under multiplication.

Every hyperbolic valued function $F: X \rightarrow \D$ defined in $X\subset \D$ is determined by its components:
\begin{equation}
\label{idem}
F(x) = F_1(x)\e + F_2(x)\e,\ x \in X,
\end{equation}
where $F_1$ and $F_2$ are real valued functions. Formula (\ref{idem}) is called the idempotent representation of $F$.

An hyperbolic metric space is a pair $(X, D)$, where $X$ is a no-empty set and $D$ is an hyperbolic positive valued function such that  for any $x, y, z \in X$ the following holds:
\begin{enumerate}[1. ]
\item $D(x, y) = 0$ if and only if $x = y$.
\item $D(x, y) = D(y, x)$.
\item $D(x, z) \preceq D(x, y) + D(y, z)$.
\end{enumerate}

The much familiar hyperbolic metric space is $(\D, D_{\k})$ where $D_{\k}$ is defined for all $\xi, \chi \in \D$; $\xi = x_1\e + x_2\ed$ and $\chi = y_1\e + y_2\ed$, such that
\begin{displaymath}
D_{\k}(\xi, \chi) := |x_1 - y_1|\e + |x_2 - y_2|\ed.
\end{displaymath}
An hyperbolic metric space is called complete if every Cauchy's sequence of hyperbolic numbers is convergent in the sense of the hyperbolic metric, see \cite{tebo}.

We say that the function $F:(X, D_1) \rightarrow (Y, D_2)$ is continuous in $\xi_0$, if for every $\epsilon \in \D^+$, there exists a $\delta \in \D^+$ such that for every $\xi \in \D$ with $D_1(\xi, \xi_0) \preceq \delta$, then 
\begin{displaymath}
D_2(F(\xi), F(\xi_0)) \preceq \epsilon.
\end{displaymath}
When both spaces $(X, D_1),  (Y, D_2)$ are equal to $(\D, D_{\k})$, the hyperbolic continuity implies the real one over $(\R^2, d_u)$, where $d_u$ is the standard metric.

A function $F:(X, D_1) \rightarrow (Y, D_2)$ defined between two hyperbolic metric spaces is called a contraction function if there exists $\kappa \in [0, \wt{1})_{\k}$ such that for every $x, y \in X$,
\begin{displaymath}
D_2(F(x), F(y)) \preceq \kappa D_1(x, y).
\end{displaymath}

Let $F:(\D, D_{\k}) \rightarrow (\D, D_{\k})$, we say that $F$ is derivable in $\xi_0 \in \D$, if there exists the limit
\begin{equation}
\label{limhol}
F'(\xi_0) :=\lim\limits_{\psi \to 0 \atop \psi \not \in \frakG_0} \frac{{F(\xi_0 + \psi) - F(\psi)}}{\psi}.
\end{equation}

In \cite{vignaux} was proved that every hyperbolic derivable function $F = u + \k v$ in the base $\{1, \k\}$ satisfies the Cauchy-Riemann type system 
\begin{displaymath}
\frac{\partial u}{\partial x} (\xi_0) = \frac{\partial v}{\partial y}(\xi_0) \text{ and } \frac{\partial u}{\partial y}(\xi_0) = \frac{\partial v}{\partial x}(\xi_0).
\end{displaymath}

It is worth noting that the last assertion implies that given a derivable function $F = F_1\e + F_2\ed$, if $\xi_0 = x_{0, 1}\e + x_{0, 2}\ed$, then
\begin{displaymath}
F(\xi_0) = F_1(x_{0,1})\e + F_2(x_{0,2})\ed.
\end{displaymath}

On the other hand, in \cite{tebo2} was shown that \ref{limhol} implies that the derivative of $F = F_1\e + F_2\ed$ is given by
\begin{displaymath}
F'(\xi_0) = \frac{\partial F_1}{\partial x_{0,1}}(\xi_0)\e + \frac{\partial F_2}{\partial x_{0,2}}(\xi_0)\ed.
\end{displaymath}

We follow \cite{vignaux} in assuming that the partial derivatives of $F$ are indeed total real derivatives and depend only of their respective components, i.e., 
\begin{displaymath}
F'(\xi_0) = \frac{d F_1}{d x_{0,1}}(x_{0, 1})\e + \frac{d F_2}{d x_{0,2}}(x_{0, 2})\ed.
\end{displaymath}

\section{HYPERBOLIC VALUED PROBABILITIES\\ AND ENTROPY}
\label{hypprobent}
Let  ${\mathcal P}= (p_1, p_2,......, p_m)$ be a finite discrete probability distribution, that is, suppose that $p_j \geq 0$, $j=1,\cdots, m$ and $\sum^{m}_{j=1}p_j=1$. Let $\mathcal F = \{f_1, ..., f_m\}$ be a   sample description space of an experiment relates to ${\mathcal P}$ in the way that the probability that occurs $f_j$ is $p_j$.

Entropy is a concept that measures how many information we can get from a probability distribution. An equivalent interpretation is given by Rényi in \cite{renyi}, where he says that entropy is the amount of uncertainty concerning the outcome of an experiment. It is measured by the quantity
\begin{displaymath}
H(\mathcal P) = -\sum_{j = 1}^{m} p_j \log p_j = \sum_{j = 1}^{m} p_j \log \frac{1}{p_j}.
\end{displaymath}

This formula was introduced by Shannon in \cite{shannon}. He established that a function to measure the entropy must fulfill with 
\begin{enumerate}
\item $H$ should be continuous in every $p_j \in \mathcal P$.
\item If all $p_j$ are equal, then $H$ should be a monotonic increasing function.
\item Given another set of probabilities $\mathcal Q$, then $H({\mathcal P}{\mathcal Q}) = H(\mathcal P) + H(\mathcal Q)$.
\end{enumerate}
One function having theses properties is called a entropy function. Later R\'enyi gave a much more simple proof (see \cite{renyi}) that just logarithm function has these properties by the next lemma.
\begin{lemma}[\cite{erdos}]
\label{lm:entropy}
Let $f:\N \rightarrow \R$ an additive number-theoretical function, that is
\begin{displaymath}
f(nm) = f(n) + f(m),
\end{displaymath}
for every $n, m \in \N$, and
\begin{displaymath}
\lim\limits_{n \to \infty} (f(n + 1) - f(n)) = 0.
\end{displaymath}
Then we have 
\begin{displaymath}
f(n) = c \log n
\end{displaymath}
where $c$ is a constant.
\end{lemma}

\subsection{Hyperbolic probability}
In \cite{eluna} was introduced the notion of a probabilistic measure which takes values in hyperbolic numbers. They show that this new measure satisfies the usual properties of a probability. These ideas led to propose the hyperbolic concept of entropy if one consider a finite discrete hyperbolic probability distribution. 
\begin{definition}\label{hpd}
Let $\frakP = \{ \rho_1, ..., \rho_n\}$ be a finite set of hyperbolic numbers of $[0, \wt {1}]_{\k}$ given in idempotent representation $\rho_k := p_{k, 1}\e + p_{k, 2}\ed$. We say that $\frakP$ is an hyperbolic probability distribution, if we have one of these states
\begin{enumerate}[(1)]
\item $\displaystyle \sum_{k = 1}^{n} \rho_k = \wt{1},$
\item $\displaystyle \sum_{k = 1}^{n} \rho_k = 1\e,$
\item $\displaystyle \sum_{k = 1}^{n} \rho_k = 1\ed.$
\end{enumerate}
\end{definition}
\begin{remark}
Cases (2) and (3) imply that for every $k \in \{1, ..., n\}$, either $\rho_k = p\e$ or $\rho_k = p\ed$ where $p \in \R$. So $\rho_k$ is a zero divisor.  
\end{remark}

The main difference from the case of real probability distribution is that in general a set of hyperbolic probabilities can not be totally ordered, it is necessary to make additional assumptions, as was mentioned in \cite{cantls}.

Following the notion of product space probability, there exists a simple interpretation of hyperbolic valued probabilities. Indeed, an hyperbolic valued probabilities $\rho_k = p_{k, 1}\e + p_{k, 2}\ed$. could be identify with 
\begin{equation}
\label{probreal}
p_{k}:=\frac{p_{k, 1} + p_{k, 2}}{2},
\end{equation}
which is the usual accumulate probability in the real sense. An alternative interpretation will be explain in Section \ref{hypchaos}. 
\subsection{Hyperbolic Entropy}
We begin with a weak hyperbolic extension of the concept of Shannon entropy using the hyperbolic correspondence (see \ref{correspondence}) of real logarithm and the accumulated probabilities \ref{probreal}.
\begin{definition}
Let $\frakP= \{ \rho_1, ..., \rho_n \}$ be a hyperbolic probability distribution. The weak hyperbolic entropy associated to $\frakP$ is defined as 
\begin{displaymath}
{\bf H}_{w}(\frakP) :=  \sum_{k = 1}^{n} -\rho_k \wt{\log}\left(\frac{p_{k, 1}+ p_{k, 2}}{2}\right) =
\end{displaymath}
\begin{displaymath}
\sum_{k = 1}^{n} -\rho_k\left(\log\left(\frac{p_{k, 1} + p_{k, 2}}{2}\right) \e + \log\left(\frac{p_{k, 1} + p_{k, 2}}{2}\right)\ed\right).
\end{displaymath}
\end{definition}
Although weak Shannon entropy can be clearly seen as generalization, we realize that the functional form of the weak entropy expression resembles the Shannon entropy and does not change the measure in essence. 

We therefore propose to define an strong hyperbolic extension of the Shannon entropy, where the role played in the traditional theory by the logarithmic function, is now played by the well-defined holomorphic logarithm function in the hyperbolic theory, to be denoted in this work by $\Log_{\D}$. See \cite[Subsection 6.5]{eluna} for more details.

Firstly, let us state the hyperbolic version of Lemma \ref{lm:entropy}, whose proof follows by making appeal to the idempotent representation of the hyperbolic logarithm function.
\begin{lemma}
Let the hyperbolic valued function $F:\N \rightarrow \D$, such that 
\begin{displaymath}
F(mn) = F(m) + F(n)  
\end{displaymath}
and
\begin{displaymath}
\lim\limits_{n \to \infty} (F(n + 1) - F(n)) = 0,
\end{displaymath}
for every $n, m \in \N$. Then there exists $\kappa \in \D$ so that $F$ has the form
\begin{displaymath}
F(\N) = \kappa \Log_{\D}(\N).
\end{displaymath}
\end{lemma}
In trying to define a strong extension of Shannon entropy to hyperbolic numbers plane, it is to be expected that the considered hyperbolic function satisfies related hyperbolic versions of the aforementioned axioms of Shannon entropy. To this aim, we propose an alternative axiomatization of the concept of strong hyperbolic entropy, which relaxed the continuity property and assuming the holomorphic condition to the hyperbolic entropy measure. This generalized entropy is defined as

\begin{definition}
The strong hyperbolic entropy associated to $\frak P$ is defined as 
\begin{displaymath}
{\bf H}_{s}(\frakP) := \sum_{k = 1}^{n} -\rho_k \Log_{\D}(\rho_k) = \sum_{k = 1}^{n} -\rho_k (\log(p_{k, 1}) \e + \log(p_{k, 2})\ed).
\end{displaymath}
\end{definition}

\begin{remark}
As the Shannon entropy measure, both hyperbolic extensions have fundamental properties which legitimate it as reasonable measure of choice. For instance: 
\begin{itemize}
\item Only when we are certain of the outcome does the measures vanish. Otherwise they are positive.
\item Both measures ${\bf H}_{w}$ and ${\bf H}_{s}$ are maximum and equal to one of theses values $log (n), log (n)e_1, log (n)e_2$ or $\Log_{\D}(n), \Log_{\D}(n)e_1, \Log_{\D}(n)e_2$ respectively, when all the $\rho_k$ are equal (e.g., $\displaystyle\frac{1}{n}, \displaystyle\frac{1}{n}e_1, \displaystyle\frac{1}{n}e_2)$, which is again the most uncertain situation.
\end{itemize}
\end{remark}

\section{HYPERBOLIC CHAOS GAME}
\label{hypchaos}
To make our exposition self-contained we give a brief sketch of the chaos game algorithm of creating a fractal in euclidean spaces (see \cite{barnsley}). The algorithm consists in the iteratively creation of a sequence of points, using a polygon and starting with a base random point inside it. Each point in the sequence is a given fraction of the distance between the previous point and one of the vertices of the polygon, which is chosen at random in each iteration.

In addition to illustrating how the algorithm work in practice, let us provide a well-known example
\begin{example}
\label{points}
Let a triangle with vertices at $a_1 = (0, 1)$, $a_2 = \left(-\displaystyle\frac{1}{2}, \displaystyle\frac{\sqrt{3}}{2}\right)$, $a_3 = \left(-\displaystyle\frac{1}{2}, -\displaystyle\frac{\sqrt{3}}{2} \right)$ and let $a_0 = (0, 0)$ be the basic point. Fig. \ref{fg:triang} illustrates the algorithm procedure for performing the fractal and some of its iterations. The result is an approximation of the well known Sierpinski Triangle.
\end{example}

\begin{figure}[ht]
\centering
\includegraphics[scale=1]{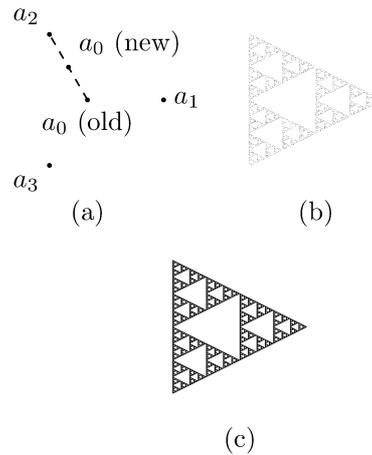}
\caption{(a)- Process to select the new base point. (b)- Outcome of five thousand of iterations. (c)- Outcome of one million of iterations.}
\label{fg:triang}
\end{figure}
To deals with general metric spaces, Barnslet \cite{barnsley} appealed to the concept of IFS with probabilities on metric spaces getting a generalization of the chaos game algorithm. For a treatment of a more general case we refer the reader to \cite{bavi}.

Here we set up some notation and terminology for completeness. An IFS on a complete metric space $(X, d)$ is a finite collection $\{f_1, ..., f_m \}$ of contractions over $X$, see \cite{barnsley}.

Let $\calH(X)$ the collection of every nonempty and compact subsets of $X$, so the Hutchinson operator $J:\calH(X) \rightarrow \calH(X)$ is defined over subsets $A \in \calH(X)$ as
\begin{displaymath}
J(A) = \bigcup_{j = 1}^{m} f_j(A).
\end{displaymath}

Due that $({\cal H}(X), h)$ is a complete metric space, when $h$ is the Hausdorff metric induced by the metric $d$, see \cite[Sec. 2.6]{barnsley}, the operator $J$ is a contraction and by Banach's fixed point theorem there exists a set ${\cal A} \in \cal{H}(X)$  that is the unique fixed point of $J$. This fixed point will be called the attractor of the IFS.

The chaos game algorithm for generating a point-set approximation of the IFS attractor works as follow: Let an IFS on a complete metric space $(X, d)$ with probabilities, i.e., an IFS $\{f_1, ..., f_m\}$ and a finite discrete probability distribution $\{p_1, ..., p_m\}$ such that each contraction $f_j$ is associated with $p_j$. Choose a base point $x_0 \in X$ and select contraction $f_j$ at random according to the probabilities $p_j$ and compute $x_{i+1} = f_j(x_i)$. Return to the second step and repeat the process iteratively.

To exemplify this algorithm, let us consider again the Sierpinski Triangle. Let $X = \R^2$ and $d$ the usual metric. The contractions are
\begin{displaymath}
f_1(x) = \frac{1}{2}x,\hspace{3mm} f_2(x) = \frac{1}{2}x + \left( \frac{1}{2}, 0 \right),
\end{displaymath}
\begin{displaymath}
f_3(x) = \frac{1}{2}x + \left(\frac{1}{4}, \frac{1}{2} \right),
\end{displaymath}
and all the probabilities $p_j,\  j= 1,2,3$ must be equal to $\displaystyle\frac{1}{3}$.

\subsection{Hyperbolic chaos game}
\label{sc:chaos}

It has already been remarked that the standard chaos game for a complete metric space only require an IFS with probabilities, and because both concepts were extended to hyperbolic metric spaces so it appear feasible to adapt the algorithm to the hyperbolic frame.
\begin{definition}[{\bf Hyperbolic chaos game}]
\label{defHypCG}
Let be $(X, D)$ a complete hyperbolic metric space and let $\calF = \{F_1, ..., F_n\}$ be a finite set of hyperbolic contractions. Set $\frakP$ an hyperbolic probability distribution (see Definition \ref{hpd}). We associate to every $F_k$ the real probability $p_k$ given by (\ref{probreal}). The algorithm runs taking a starting point $x_0 \in X$ and then choose randomly the contraction $F_{k} \in \calF$ on each iteration accordingly to the assigned probability $p_{k}$. The rest of the algorithm runs as before.
\end{definition}
If the hyperbolic probabilities are zero divisors, then we only have to consider the natural identification with $\R$.

\begin{remark}
Attractors of contractive hyperbolic IFS on complete hyperbolic metric spaces may not be unique because the respective $\D$-Hausdorff distance is not a metric, see \cite{tebo2}. Consequently, the hyperbolic chaos game represents a method of generating a representative element of the attractor class. 
\end{remark}

\subsection{$\D$-{chaos game}}
Now, we focus our study to modify the chaos game on the hyperbolic metric space $(\D, D_{\k})$. More precisely, treating the ordinary hyperbolic probability distribution as true hyperbolic numbers rather than real identifications.

\vspace{3mm}
{\noindent \bf Algorithm:} Let $\calF=\{F_1, ..., F_n\}$ be a finite set of contractions over $(\D, D_{\k})$ where $F_k = F_{k, 1}\e + F_{k, 2}\ed$ and let $\frakP$ be an hyperbolic probability distribution associated to $\calF$. The sets of idempotent components of $\calF$ will be denoted by $\Phi_1 = \{F_{1,1}, ..., F_{n, 1}\}$ and $\Phi_2 = \{F_{1, 2}, ..., F_{n, 2}\}$.

\begin{enumerate}[\hspace{3mm}1. ]
\item Chose a base point $\xi_0 \in \D$.
\item Randomly select two function $F_{s, 1} \in \Phi_1$ and $F_{t, 2} \in \Phi_2$, according to the real idempotent components $p_{s, 1}$ and $p_{t, 2}$ of the probabilities $\rho_s, \rho_t \in \frakP$ respectively and define the function $G_{s,t} = F_{s, 1}\e + F_{t, 2}\ed$.
\item Subsequently compute the next point as $\xi_{i + 1} = G_{s,t}(\xi_i)$.
\item The process is repeated from the step two iteratively a finite set of times.
\end{enumerate}

A natural question to ask is whether functions $G_{s,t}$  are also contractions. The answer is given by the next lemma.

\begin{lemma}
Let $F_1 = F_{1, 1}\e + F_{1, 2}\ed$ and $F_2 = F_{2, 1}\e + F_{2, 2}\ed$ be two hyperbolic contractions over $(\D, D_{\k})$ with contractive factors $\kappa_1 = c_{1, 1}\e + c_{1, 2}\ed$ and $\kappa_2 = c_{2, 1}\e + c_{2, 2}\ed$ respectively. Then function $G_{s,t} = F_{s, 1}\e + F_{t, 2}\ed$ is again a contraction over $(\D, D_{\k})$ and its contractible factor is $\kappa = c_{s, 1}\e + c_{t, 2}\ed$.
\end{lemma}
\begin{proof}
The proof is followed at once because in \cite[Def. 2.8]{tebo2} every contraction is defined in a point-wise way with the hyperbolic partial order.
\end{proof}


\begin{example}
\label{exDCG}
Let the IFS generates the Sierpinski Triangle in the Fig. \ref{fg:triang} and identify every contraction with its respective hyperbolic version, which are affine hyperbolic transformation (see \cite[Sec. 2.3]{tebo2}). Then we have the hyperbolic IFS with probabilities $\{(\D, D_{\k}), F_1, F_2, F_3 \}$ where
\begin{equation}
\label{hypeftriang}
\begin{split}
F_1(\xi) & = \wt{\frac{1}{2}}\xi, \\
F_2(\xi) & = \wt{\frac{1}{2}}\xi + \frac{1}{4}\e + \frac{1}{2}\ed, \\
F_3(\xi) & = \wt{\frac{1}{2}}\xi + \frac{1}{2}\e,
\end{split}
\end{equation}
and for every $k \in \{1, 2, 3\}$ the probability $\rho_k$ is equal to $\wt{\displaystyle\frac{1}{3}}$.
\end{example}

One approximation of this hyperbolic IFS with the hyperbolic chaos game is shown in the Fig. \ref{fig:iths}.

\begin{figure}[ht]
\centering
\includegraphics[scale=0.6]{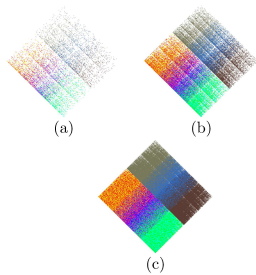}
\caption{Representation of the result to apply the hyperbolic chaos game to IFS \ref{hypeftriang} with $\rho_k=\wt{\displaystyle\frac{1}{3}}$. (a)- 10000 iterations. (b)- 100000 iterations, (c)- 1000000 iterations.}
\label{fig:iths}
\end{figure}

\begin{example}
Now we change the hyperbolic probability distribution in the Example \ref{exDCG} by
\begin{equation}\label{hpd2}
\begin{split}
\rho_1 & = \frac{1}{10}\e + \frac{1}{4}\ed, \\
\rho_2 & = \frac{3}{10}\e + \frac{1}{5}\ed, \\
\rho_3 & = \frac{3}{5}\e + \frac{11}{20}\ed,
\end{split}
\end{equation}
\end{example}
Note that we get a similar figure when we repeat the process a big quantity of times, but it has different accumulation of points in $[0, \wt{1}]_{\k}$. See Fig. \ref{fig:iths2}.

\begin{figure}[ht]
\centering
\includegraphics[scale=0.6]{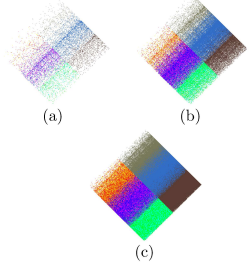}
\caption{Representation of the result to apply the $\D$-{chaos game} to IFS \ref{hypeftriang} with the hyperbolic probability distribution (\ref{hpd2}). (a)- 10000 iterations. (b)- 100000 iterations, (c)- 1000000 iterations.}
\label{fig:iths2}
\end{figure}

The similitude is due that in the hyperbolic case the attractor associated to an IFS is an equivalence class of no empty compact sets (see \cite[Sec. 3.1]{tebo2}).

Given ${\cal F}=\{F_1, ..., F_n\}$ an hyperbolic IFS over $(\D, D_{\k})$ with $\frakP$ an hyperbolic probability distribution associated to $\calF$, the number of possible combination functions $\{G_{s,t}\}$ generated by the $\D$-{chaos game} is $\{1, ..., n^2\}$ and the real probabilities considering the selection of the idempotent components of the contractions like independent events are ${\pi}_m = p_{s, 1} p_{t, 2}$ with $m \in \{1, ..., n^2\}$. Let $\wt{\frakQ}$ denotes the hyperbolic probability distribution $\{\wt{{\pi}_1},....., \wt{{\pi}_{n^2}}\}.$ 

In an alternative way, we will use the symbol $\frakK$ to denote the hyperbolic probability distribution $\{{\omega}_1,....., {\omega}_{n^2}\}$ , where ${\omega}_m =\wt{\displaystyle\frac{1}{n}} \left(p_{s, 1}\e + p_{t, 2}\ed\right).$  

Because the strong entropy is measuring the information in every component of the system with its probabilities associated the following relations hold.

\begin{displaymath}
{\bf H}_s(\frakP) \preceq \wt{H(\wt{\frakQ})},
\end{displaymath}

\begin{displaymath}
{\bf H}_s(\frakP) \preceq {\bf H}_s(\frakK).
\end{displaymath}

\section{FURTHER REMARKS}
Repeating the $\D$-{chaos game} to an IFS, a limited number of times, it is created a sequence of points in the interior of the hyperbolic interval 
$[0, \wt{1}]_{\k}$, which may be viewed as a rectangle in the Euclidean plane. These points fit into the pattern of those noise points in one image. If we take now an image (looking it as the rectangle) and we deform, by an affine transformation, until it be embedded into interval $[0, \wt{1}]_{\k}$, we can overlap the sequence of points in it, so in this overlapping a noise pattern in the image emerge. This suggests a natural question: whether the noise in an image could be removed by using the $\D$-{chaos game}. Some partial evidence supports the conjecture that the use of an analogous of the collage theorem (see \cite{barnsley}) for $\D$-{chaos game}, together with some colored technique, introduced in \cite{barnsley2}, would be used to remove noise from an image. The noise pattern is located within the image using the collage theorem to obtain an IFS. The D-chaos game is then applied over it to replace the noise pattern with appropriately colored points.

\section*{Acknowledgements}
This research was partially supported by Consejo Nacional de Ciencia y Tecnolog\'ia (CONACYT) and by Instituto Polit\'ecnico Nacional in the framework of SIP programs.

\end{document}